\documentclass[11pt]{amsart}
\usepackage{amsfonts,amsmath,mathrsfs,amssymb,amsthm,amscd,latexsym}

\usepackage[usenames]{color}
\usepackage[all]{xy}
\usepackage{enumerate}
\xyoption{all}
\usepackage{srcltx} %per fer cerca inversa
\newtheorem{thm}{Theorem}[section]
\newtheorem{lem}[thm]{Lemma}
\newtheorem{cor}[thm]{Corollary}

\theoremstyle{definition}
\newtheorem{rem}[thm]{Remark}

\newcommand{\lra}{\longrightarrow}

\title{Hyperelliptic Jacobians and isogenies}

\author{J.C. Naranjo}
 \address{Juan Carlos Naranjo  \\Universitat de Barcelona  \\ Departament de Matem\`atiques i Inform\`atica \\ Gran Via 585  \\ 08007 Barcelona, Spain  }
 \email{jcnaranjo@ub.edu}

 \author{G.P. Pirola}
 \address{Gian Pietro Pirola  \\ Universit\`a degli Studi di Pavia  \\ Dipartimento di Matematica \\ Via Ferrata, 1  \\ 27100 Pavia, Italy  }
 \email{gianpietro.pirola@unipv.it}

 \thanks{The first named authot was partially supported by the Proyecto de Investigaci\'on MTM2015-65361-P.  The second named author is  partially supported by INdAM (GNSAGA); PRIN 2015 \emph{``  Moduli spaces and Lie theory''} and by  (MIUR): Dipartimenti di Eccellenza Program (2018-2022) - Dept. of Math. Univ. of Pavia.}
\begin{document}

\begin{abstract}
In this note we mainly consider abelian varieties isogenous to hyperelliptic Jacobians.
In the first part we prove that a very general hyperelliptic Jacobian of genus $g\ge 4$ is not isogenous to a non-hyperelliptic Jacobian. As a consequence we obtain that the intermediate Jacobian of a very general cubic threefold is not isogenous to a Jacobian. Another corollary tells that the Jacobian of a very general $d$-gonal curve of genus $g \ge 4$ is not isogenous to a different Jacobian.

In the second part we consider a closed subvariety $\mathcal Y \subset \mathcal A_g$ of the moduli space of principally polarized varieties of dimension $g\ge 3$. We show that if a very general element of $\mathcal Y$ is dominated by the Jacobian of a curve $C$ and $\dim \mathcal Y\ge 2g$, then $C$ is not hyperelliptic. In particular, if the general element in $\mathcal Y$ is simple, its Kummer variety does not contain rational curves. Finally we show that a closed subvariety $\mathcal Y\subset \mathcal M_g$ of dimension $2g-1$ such that the Jacobian of a very general element of $\mathcal Y$ is dominated by a hyperelliptic Jacobian is contained either in the hyperelliptic or in the trigonal locus.
\end{abstract}

\maketitle

\pagestyle{plain}
\section{Introduction}
Hyperelliptic Jacobians of dimension $g\ge 3$ are for many reasons very special irreducible principally polarized abelian varieties (ppav in the sequel).  First of all they have
the highest dimension of the singular locus of the Theta divisor $\Theta$. Indeed, it is well-known that the dimension of $Sing (\Theta )$ is $g-3$ and by a Theorem of Ein and Lazarsfeld (see \cite{EL}) only reducible ppav have dimension $g-2$. It was proved by Andreotti and Mayer in \cite{AM} that the hyperelliptic locus is an irreducible component of the locus of ppav with this property, and it is expected to be the component of highest dimension (see for instance \cite[Conjecture 9.1]{CvdG} and \cite[section 5]{Gr}).  Moreover hyperelliptic Jacobians  have the highest degree of the Gauss map from  $\Theta $ to the projectivization $ \mathbb P^{g-1}$ of the tangent space at the origin (see \cite[Chaper VI, section 3]{ACGH} and \cite{Verra}). 
Finally its Kummer variety contains countably many rational curves (see \cite{kummer}). Therefore the Kummer variety of an abelian variety isogenous to a hyperelliptic Jacobian would have also this rather special property.

It is worth to notice the remarkable result of Mestre (see \cite{mestre}): there exist two families  $\mathcal C,\mathcal C' \to \mathbb A^{g+1}$ of hyperelliptic curves of genus $g$ such that the images of both families in $\mathcal M_g$ are different and of dimension exactly $g+1$, and there exists a non-degenerate $(2,2)$ correspondence in $\mathcal C\times_{\mathbb A^{g+1}} \mathcal C'$. In particular the Jacobian of a general element of the first family is isogenous to the Jacobian of a curve of the second family. Our main purpose goes in the opposite direction: we want to prove that the Jacobian of a very general hyperelliptic curve is not isogenous to a distinct Jacobian. 

In the paper \cite{isog}  we prove that the Jacobian of a general element of a low codimension subvariety of $\mathcal M_g$  cannot be isogenous to a distinct Jacobian. This had been proved by Bardelli and Pirola for the whole moduli space when $g\ge 4$ (see \cite{bardelli_pirola}). Recently some of these results have been proved for Prym varieties (see \cite{laface_martinez}). In this paper we consider abelian varieties and Jacobians which are isogenous to, or even more generally dominated by, hyperelliptic Jacobians. In the first part of the paper we show:

\vskip 3mm
\begin{thm}\label{jacs_isog_hyp}
Let $\mathcal H_g\subset \mathcal M_g$ be the locus of hyperelliptic curves of genus $g$. Assume that for a very general $C\in\mathcal H_g$ there exists an isogeny $JD \lra JC$, where $D$ is a smooth curve. Assume also that either $D$ has genus $g\ge 4$, or is a hyperelliptic curve of genus $3$. Then $D=C$ and the map is the multiplication by a non-zero integer.
\end{thm}

Notice that this theorem contains two statements: the first refers to isogenies between hyperelliptic Jacobians (with $g\ge 3$) and it is an extension of the main theorem in \cite{bardelli_pirola}. The second says that if $g\ge 4$ a very general hyperelliptic Jacobian is not isogenous to a non-hyperelliptic Jacobian. Observe that this second result does not hold for $g=3$: indeed the set of irreducible Jacobians isogenous to a fixed hyperelliptic Jacobian is dense in $\mathcal M_3=\mathcal A_3$ since it corresponds to the action of $Sp(\mathbb Q, 6)$ on the Siegel upper space $\mathcal H_3$.

We obtain two consequences of this theorem. The first one concerns the intermediate Jacobian of a smooth cubic threefold:

\begin{thm}\label{thm_cubic} 
 Let $X$ be a very general cubic threefold, then $JX$ is not isogenous to a Jacobian.  
\end{thm}

In a paper of Voisin (see \cite{Voisin}) about the universal triviality of the Chow group of $0$-cycles of cubic hypersurfaces she uses isogenies of intermediate Jacobians of some cubic threefolds to Jacobians of curves in order to produce curves representing odd minimal cohomological classes. Theorem \ref{thm_cubic} says that her method cannot be extended to the general cubic threefold.

The method of proof is by contradiction: assuming the existence of a family of isogenies we move to the boundary of the locus of intermediate Jacobians. Following \cite{collino_fano} we can specialize to the locus of hyperelliptic Jacobians and then we are led to the study of the Jacobians which are isogenous to hyperelliptic Jacobians. Then we finish by applying Theorem \ref{jacs_isog_hyp}.

We  also deduce with a similar method the following result of independent interest:

\begin{thm}\label{gonal}
 Let $\mathcal G_{d,g} \subset  \mathcal M_g$ be the locus of the $d$-gonal curves of genus $g\ge 4$. Let $C$ be a very general curve in any component of $\mathcal G_{d,g}$, and let $f:JD\lra JC$ be an isogeny, then $D=C$ and $f=n \cdot Id_{JC}$ for some $n\in \mathbb Z\setminus \{0\}$.
\end{thm}

We observe that this is not a consequence of the main theorem in \cite{isog} since the codimension of the $d$-gonal locus does not satisfy in general the numerical restriction on the codimension of that theorem ($g-2$ versus $\frac {g-4}3$). 

\vskip 3mm
Motivated by these problems we pose the question of how big can be a locus of ppav which are dominated by  hyperelliptic Jacobians. We find the following answer:

\vskip 3mm
\begin{thm}\label{main}
 Let $\mathcal Y\subset \mathcal A_g$ be a closed irreducible subvariety with $\dim \mathcal Y \ge 2g$ and $g\ge 3$. Assume that for a very general $A$ in $\mathcal Y$ there exists a dominant map $JC \lra A$, where $C$ is a smooth irreducible curve. Then $C$ is not hyperelliptic.
\end{thm}

By taking $\mathcal Y=\mathcal H_g$, the hyperelliptic locus, we obtain that the bound on the dimension is sharp. A less simple example is given in Remark \ref{example}. The method of proof is completely different since we use deformation arguments, the key tool being the adjunction procedure of two holomorphic forms on a curve killed by an infinitesimal deformation as explained in $\ref{CoPi}$. 

 It is natural to ask whether this result can be generalized to higher gonality, that is, how big can be an irreducible  closed $\mathcal Y\subset \mathcal A_g$ with the property that a very general element of $\mathcal Y$ is dominated by the Jacobian of a curve of gonality $d$? The reader can see similar questions and some conjectures in a recent preprint of Voisin \cite{Voisin2} (see, for example, Conjecture 0.2 in loc. cit.).

\vskip 3mm
A nice immediate consequence of our result is the following generalization of the main theorem in \cite{kummer}:
\begin{cor}
  Let $\mathcal Y\subset \mathcal A_g$ be a closed irreducible subvariety with $\dim \mathcal Y \ge 2g$ and $g\ge 3$. Assume that a general element $A$ in $\mathcal Y$ is simple. Then the Kummer variety of $A$ does not contain rational curves.
\end{cor}

\vskip 3mm
Next we consider Jacobians dominated by hyperelliptic Jacobians which turns out to be a much more intrincate situation. Observe that this is equivalent to ask which Jacobians contain hyperelliptic curves. The existence of curves of low genus or fixed gonality in a general Jacobian or in a Prym variety has been considered in previous papers (see \cite{Egloffstein} and \cite{NP_indagationes}). 

We obtain the following result:

\vskip 3mm
\begin{thm} \label{jacs_dom_hyper}
 Let $\mathcal Y\subset \mathcal M_g$ be a closed irreducible subvariety of dimension $2g-1$, $g\ge 5$, such that for a very general curve $D$ in $\mathcal Y$ there exists a dominant map of Jacobians $JC\lra JD$, where $C$ is hyperelliptic. Then either $\mathcal Y$ is the hyperelliptic locus in $\mathcal M_g$ or it is contained in the locus of trigonal curves. 
\end{thm}
 
The proof of this theorem is more involved and uses that for a curve of Clifford index $\ge 2$ the infinitesimal deformations of rank $1$ are the Schiffer variations (see \cite{Gr_IVHS_III}). The quintic plane case is ruled out by using that there are no rank $1$ infinitesimal deformations preserving the planarity of the curve (see \cite{quintic}). 

\vskip 3mm
The structure of the paper is as follows: in section $2$ we prove Theorem \ref{jacs_isog_hyp} by using degeneration methods, the techniques in this part are very much connected with those of \cite{isog}. Next we deduce   Theorems  \ref{thm_cubic} and \ref{gonal}. Section $3$ is devoted to prove Theorem \ref{main}. This time the proof involves the adjoint construction developed in \cite{collino_pirola}. Following the same ideas we prove  Theorem \ref{jacs_dom_hyper} in the last section.  
 
\vskip 3mm
We work over the complex numbers. As a general rule the hyperelliptic curves are denoted with the letter $C$. We say that a property holds for a very general point of a variety $X$ if it holds in the complement of the union of countably many proper subvarieties of $X$. 

\vskip 5mm
\section{Jacobians isogenous to hyperelliptic Jacobians and applications} 

\vskip 3mm

 The aim of this section is to give the proof of Theorem \ref{jacs_isog_hyp} and next to deduce Theorem \ref{thm_cubic} and Theorem \ref{gonal}. 
 
 \begin{rem}\label{construction} 
 The beginning of the argument is a standard construction that will be used several times in the sequel: 
 Assume that $\mathcal Y$ is a closed
 irreducible subvariety of some moduli space $\mathcal M$ (of curves, of cubic threefolds, of principally polarized abelian varieties,...). In all the examples that we consider, an element $y\in \mathcal Y $ has attached a ppav $A_y$ (the Jacobian, the Intermediate Jacobian,...). Assume that for a very general element $y\in \mathcal Y$ the abelian variety $A_y$ is dominated by the Jacobian of a smooth curve. Then there exists a family of surjective  maps of abelian varieties:   
\begin{equation}\label{map_of_families}
 \xymatrix@C=1.pc@R=1.8pc{
\mathcal {JD}  \ar[rd] \ar[rr]^ f && \mathcal {A} \ar[ld] \\
&\mathcal U 
}
\end{equation}
over a quasi-projective variety $\mathcal U$ and we can assume that the map provided by the existence of the moduli functor associated to the family $\mathcal A\to \mathcal U$:
\[
 \Phi: \mathcal U \lra \mathcal  M
\]
induces a generically finite dominant map $\mathcal U \lra \mathcal Y$.
 \end{rem}
 
 \begin{proof} (of Theorem \ref{jacs_isog_hyp})
 We assume, as in the statement of the Theorem, that the Jacobian of a very general hyperelliptic curve $C$ of genus $g$ is isogenous to another Jacobian $JD$.
We are assuming either that $g\ge 4$ or that $D$ is hyperelliptic of genus  $ 3$. 
  Let   $\mathcal {JD} \to \mathcal {JC}$  be  family of  isogenies over  $\mathcal U$ provided by (\ref{map_of_families}) with  $\mathcal Y =\mathcal H_g$, and denote by 
  $\Phi:\mathcal U\lra \mathcal H_g$ the generically finite dominant map sending $y\in \mathcal U$ to the isomorphism class of $(\mathcal {JC})_y$. 

\vskip 3mm
We go to the boundary of $\mathcal H_g$ in $\overline {\mathcal M_g}$. So we assume that $C$ degenerates to a hyperelliptic nodal curve $C_0$, i.e. we consider a smooth hyperelliptic curve $\tilde C_0$ of genus $g-1$ and we identify, to form a node, two different points of the form $x, \iota_{C_0}(x)$, where $\iota_{\tilde C_0} $ is the hyperelliptic involution. It is well-known that these curves appear in the boundary of the hyperelliptic locus (see for instance \cite[Ch. 10, section 3]{ACG2}).

\vskip 3mm
\begin{rem} 
We will use several times that the limit of a family of isogenies of Jacobians is also an isogeny, that is, an \'etale surjective map between the generalized Jacobians.

Consider in fact two semistable flat families of curves parametrized by a disk  $p :\mathcal D\lra \mathbb D$ and  $\pi: \mathcal C\lra \mathbb D$ such that the curves are smooth away from the central fibres $\mathcal D_{0}$, $\mathcal C_{0}$. A  family of isogenies $\varphi: \mathcal {JD} \lra \mathcal {JC}$ over the punctured disk, gives an isomorphism of the local system
$R^1p_{*} \mathbb Q $ and $R^1\pi_{*} \mathbb Q $ and therefore the existence of a map $\varphi_0:\mathcal {JD}_0 \lra  \mathcal {JC}_0$ that is an isogeny of semiabelian varieties.

For instance for the case when  $C_0$ is irreducible with  one node,  we have only one vanishing cycle for the family $\mathcal C$ and therefore
also the stable model of the $D_0$ must have only one node (see more details in \cite[p. 267]{bardelli_pirola}).
\end{rem}

\vskip 3mm

The generalized Jacobian of $C_0$ is an extension of $J\tilde C_0$ by $\mathbb C^*$ and the limit of the family of isogenies provide a diagram as follows: 

\[
 \xymatrix@C=1.pc@R=1.8pc{
0 \ar[r]     &\mathbb C^{*\, r}\ar[r] \ar[d]^{\gamma}   & JD_{0} \ar[r] \ar[d]^{f_0}  & J\tilde {D}_0 \ar[r] \ar[d]^{\tilde f_{0} }  & 0 \\
0 \ar[r]     &\mathbb C^*     \ar[r]           & JC_{0}  \ar[r]         & J\tilde {C}_{0} \ar[r]           & 0
}
\]
where $\tilde D_{0}$ and $\tilde C_{0}$ stand for the normalizations of $D_{0}$ and $C_{0}$ respectively. Since $f_{0}$ has finite kernel,
$r$ must be $1$ 
and $\gamma (t)=t^m$ for some non-zero integer $m$. An extension as the first row of the diagram corresponds to a class $\pm [y-z]\in J\tilde D_0$ for some points $y,z$ in $\tilde D_0$.

\vskip 3mm
As in \cite{bardelli_pirola}, section 2, to compare the extension classes of each horizontal short exact sequence, we decompose the last diagram into
\[
 \xymatrix@C=1.pc@R=1.8pc{
0 \ar[r]     &\mathbb C^* \ar[r] \ar[d]^{\gamma}   & JD_{0} \ar[r] \ar[d]  \ar@/^/[dd]^<(.2){f_{0}}  & J\tilde {D}_{0} \ar[r] \ar@{=}[d]   & 0 \\
0 \ar[r]     &\mathbb C^*  \ar[r]   \ar@{=}[d]  &E \ar[r] \ar[d] &  J\tilde {D}_{0}\ar[d]^{\tilde f_{0} } \ar[r] &0 \\
0 \ar[r]     &\mathbb C^*     \ar[r]           & JC_{0}  \ar[r]         & J\tilde {C}_{0} \ar[r]           & 0.
}
\]
We get that $\rho(x):=\tilde f_0^* (x-\iota_{\tilde C_0}(x))= m( y-z)$ in $J\tilde D_0$ (we are identifying each Jacobian with its dual by using the principal polarization). Now we move $x$ in the curve $\tilde C_0$, in other words, we change the limit curve keeping fixed the normalization of the curve. Observe that as $x$ moves, the curve $\tilde D_0$ can non vary, since there do not exist non-trivial families of isogenies with a fixed target. Hence the equality above says that the map $\rho$ gives  a non-trivial map from $\tilde C_0$ to surface $ m (\tilde D_0 - \tilde D_0)$. 

\vskip 3mm
We want to show that the normalizations of the limit curves are the same, that is, that $\tilde C_0=\tilde D_0$. This implies by the genericity that the map of Jacobians is the multiplication by a constant. To prove this we proceed separately depending on the hypothesis on $D$:

\textbf{Case 1:} Assume that $D$ is hyperelliptic of genus $g\ge 3$.

Then $D_0$ is also hyperelliptic. This means that $\tilde D_0$ is a hyperelliptic curve of genus $g-1$ and that the extension class is of the form 
$[y-\iota_{\tilde D_0}(y)]$, where $\iota_{\tilde D_0}$ stands for the hyperelliptic involution. Observe that since $\tilde C_0$ is general so is $\tilde D_0$. This implies that the map:
\[
 \tilde D_0 \lra J\tilde D_0, \quad y \mapsto m(y-\iota_{\tilde D_0}(y))
\]
has degree $1$. Indeed, this is a dimension count: otherwise there is map of degree $2m$ from $\tilde D_0$ to $\mathbb P^1$ with two fibres of the form $my+mz$. Then:
\[
 2(g-1)-2=-2m+4(m-1)+r=2m-4+r,
\]
thus $r=2(g-m)$. Then the family of curves $\tilde D_0$ with this kind of pencils depend on $r+2-3=2(g-m)-1=2g-1-2m$ parameters, therefore 
$\dim \mathcal H_{g-1}=2g-3\ge 2g-1-2m$, so  $m=1$, and for $m=1$ the degree of the map is $1$ by the uniqueness of the hyperelliptic linear series. We obtain that the map $\rho $ provides an isomorphism between the curves $\tilde C_0$ and $\tilde D_0$.

\vskip 3mm
\textbf{Case 2:} Assume that $D$ has genus $g\ge 4$.

If the curve $\tilde D_0$ is hyperelliptic we can apply the case 1. So assume that it is not and we reach a contradiction. First we need the following:

\vskip 3mm
\textbf{Claim:}  The  map of surfaces $\kappa: \tilde D_0 \times \tilde D_0 \lra  m(\tilde D_0 - \tilde D_0)$ sending $(y,z)$ to $m(y-z)\in J\tilde D_0$ is birational.
\vskip 3mm
\textbf{Proof of the claim:} 
We consider the diagram of rational maps:
$$
\xymatrix@C=2cm{
              \tilde D_0 \times \tilde D_0   \ar@{-->}[dd]^{\kappa}  \ar@{-->}[dr]^{s} &  \\
	         & Grass(1,\mathbb PH^0(\tilde D_0,\omega_{\tilde D_0})^*) \\
	   m(\tilde D_0-\tilde D_0) \ar@{-->}[ur]^{G}  &       
    }
$$
where $s$ is the secant map, that is the map sending $(y,z)$ to the line generated by the images of $y$ and $z$ under the canonical map. On the other hand $G$ stands for the Gauss map in the Jacobian sending a point of the surface $m(\tilde D_0-\tilde D_0)$ to the projectivization of the tangent plane to the surface translated to the tangent space in the origin of the Jacobian. It is well-known that this diagram commutes. Moreover both maps, $s$ and $G$, have degree $2$: this is clear for $s$ since we assume that $\tilde D_0$ is not hyperelliptic. For $G$ it is a consequence of the fact that the Gauss map is invariant under isogenies, hence we can replace the surface by $\tilde D_0-\tilde D_0$. This implies that $\kappa$ is birational.
\qed

By composing $\rho $ with the inverse of $\kappa $ and taking a projection on one of the factors we obtain a non-trivial rational map $\tilde C_0 \lra D_0$. We can assume that $\tilde C_0$ is general in the hyperelliptic locus, hence $\tilde C_0 =\tilde D_0$, this contradicts that $\tilde D_0$ is not hyperelliptic.   

\vskip 3mm
The end of the proof follows closely the argument in \cite[Section 6]{isog} and \cite[Proposition 4.2.1]{bardelli_pirola}. We quickly outline the main idea for the convenience of the reader.
We go back to our family of isogenies $f:\mathcal {JD} \lra \mathcal {JC}$. Consider a general point $t\in \mathcal U$ corresponding to smooth curves $D_t$ and $C_t$. 
Observe that for all $t$ the isogeny is determined by the map at the level of homology groups 
\[
 f_{t,\mathbb Z}:H_1(D_t, \mathbb Z) \lra H_1 (C_t,\mathbb Z) 
\]
which we still denote by $f_t$. We set $\Lambda_t\subset H_1(C_t,\mathbb Z)$ for the image of $f_t$. This is a sublattice of maximal rank $2g$. If we are able to prove that
$\Lambda_t =n H_1(C_t,\mathbb Z)$ for some positive integer $n$, then we would get $D_t\cong C_t$ and $f_t$ would be multiplication by $n$.

To show the equality  $\Lambda_t =n H_1(D_t,\mathbb Z)$ we obtain information on 
$\Lambda_t$ from the homology groups of the limits $C_0$ considered above. 
We can assume that there exists a disk $\mathbb  D\subset \overline{\mathcal U}$ centered at the class of the curve $C_0$ such that the curves $D_t, C_t$ corresponding to $\mathbb D \setminus \{0\}$ are smooth. After performing a base change we can assume that there is a family of 
 isogenies $f_{\mathbb D}:\mathcal {JD}_{\mathbb D}\lra \mathcal {JC}_{\mathbb D}$ that coincides with the original isogeny $f_t$ for a general $t$.
 We call $\mathcal D_{\mathbb D}$ and  $\mathcal C_{\mathbb D}$ the corresponding families of curves. Since the central fibres $D_0$ and $C_0$ are retracts of $\mathcal D_{\mathbb D}$ and  $\mathcal C_{\mathbb D}$ respectively, we have a 
diagram   as follows:
 \[
 \xymatrix@C=1.pc@R=1.8pc{
 H_1(D_t,\mathbb Z)\ar[d]^{f_t}  \ar[r] & H_1(\mathcal D_{\mathbb D},\mathbb Z) \ar@{=}[r] & H_1(D_0,\mathbb Z) \ar[d]^{f_0} \\
 H_1(C_t,\mathbb Z)                  \ar[r] & H_1(\mathcal C_{\mathbb D},\mathbb Z)  \ar@{=}[r]& H_1(C_0,\mathbb Z).  
 }
 \]
By the previous discussion we know that $D_0=C_0$ and $f_0$ is the multiplication by a non-zero integer $n$. Then modulo the vanishing cycles (the generators of the kernels of the horizontal arrows in the last diagram) the map $f_t$ is the multiplication by this constant. By performing a second degeneration with a different vanishing cycle one easily gets that $f_t$ is the multiplication by $n$. This finishes the proof of the Theorem.
\end{proof}

Next we show that the statements on intermediate Jacobians of cubic threefolds and on Jacobians of $d$-gonal curves are corollaries of Theorem \ref{jacs_isog_hyp}.
\vskip 3mm
\begin{proof} (of Theorems \ref{thm_cubic})and \ref{gonal})
Both are similar since the hyperelliptic locus appear in the closure of the locus of intermediate Jacobians (see \cite[Theorem (0.3)]{collino_fano}) as well as in the closure of the locus of $d$-gonal curves. So we only give the proof in the first case. Proceeding by contradiction and using Remark \ref{construction}, we can assume the existence of a family of isogenies $f:\mathcal {JC}\lra \mathcal {JT}$ over an open set $\mathcal U$, where $JT_b$ is the intermediate Jacobian of a cubic threefold $T_b$ for any $b\in \mathcal U$ and $JC_b$ is the Jacobian of a curve of genus $5$. We can apply verbatim the  argument of the last part of the  previous theorem when we consider a restriction of the family to a disk in such a way that the central fiber corresponds to a limit isogeny $f_0: JC_0 \lra JT_0$. This time $JT_0$ is an actual abelian variety and in fact it is the Jacobian of a general hyperelliptic curve $D_0$ of genus $5$. We apply Theorem \ref{jacs_isog_hyp} and we get that $C_0=D_0$ and $f_0$ is the multiplication by a non-zero integer. We obtain that the same is true for the general element in the disk. In particular the intermediate Jacobian of a smooth cubic threefolds would be isomorphic, as ppav, to a Jacobian which contradicts the main Theorem in \cite{clemens_griffiths}.   \end{proof}

\vskip 5mm
\section{Abelian varieties dominated by hyperelliptic Jacobians}

We devote the whole section to the proof of Theorem \ref{main}. Since the beginning of the proof of Theorems \ref{main} and \ref{jacs_dom_hyper} is very similar we start with the common set-up. The proof of Theorem \ref{jacs_dom_hyper} will be completed in the next section. 

Let us consider a closed subvariety $\mathcal Y$ of either  $\mathcal A_g$ or $\mathcal M_g$. In the first case we assume that $\dim \mathcal Y\ge 2g$ and in the second case that $\dim \mathcal Y=2g-1$. We proceed by contradiction, so we assume that the abelian variety attached to a very general element of $\mathcal Y$ (the Jacobian of the corresponding curve when $\mathcal Y \subset \mathcal M_g$) is dominated by the Jacobian of a hyperelliptic curve. We apply the construction in Remark \ref{construction} and we obtain the existence a family of surjective maps of abelian varieties (see diagram (\ref{map_of_families}))
\[
 \xymatrix@C=1.pc@R=1.8pc{
\mathcal {JC}  \ar[rd] \ar[rr]^ f && \mathcal {A} \ar[ld] \\
&\mathcal U 
}
\]
where $C_y$ is hyperelliptic for any $y\in \mathcal U$. In the case that $\mathcal Y\subset \mathcal M_g$ the family $\mathcal A$ is in fact a family of Jacobians $\mathcal {JD}$. 
Therefore in what follows, for the case $\mathcal Y\subset \mathcal M_g$, the abelian variety $A_y$ has to be replaced by the Jacobian $JD_y$. 

We also have a generically finite dominant map $\mathcal U \lra \mathcal Y$ sending $y\in \mathcal U$ to the isomorphism class of $A_y$.
We fix a general point $y\in \mathcal U$. We set
 $\mathbb T:= T_{\mathcal U}(y)\cong T_{\mathcal Y}(\Phi(y)) $. The
 differential of $\Phi$ in $y$ gives:
\[
 d\Phi_y: \mathbb T \hookrightarrow Sym^2\, H^{1,0}(A_y)^*.
\]
On the other hand $f_y:JC_y \lra A_y$ induces an inclusion of complex vector spaces $W:=f_y^*(H^{1,0}(A_y)) \subset H^0(C_y,\omega_{C_y})$.  Let $B$ the base divisor of the linear system $| W | \subset | \omega _{C_y} | $. Observe that $W\subset H^0(C_y,\omega _{C_y}(-B))$. 
Let us define $G_W$ the subset of the Grassmannian 
\[
G:=Grass(2,H^0(C_y,\omega_{C_y}(-B))) 
\]
consisting in the $2$-dimensional subspaces contained in $W$. Among these subspaces we are interested in those which, as linear systems, have some additional base point on $C_y$. In other words:
\[
 G_{BL}=\bigcup_{p\in C_y} Grass(2, W(-p)),
\]
where $W(-p)$ is the intersection of $W$ with $H^0(C_y,\omega_{C_y}(-B-p))$.

\vskip 3mm
\begin{lem}\label{grass} The subset $G_{BL}\subset G_W$ is an irreducible divisor of $G_W$. In particular, there exists a $2$-dimensional linear subspace $V\subset W$ such that the base locus of the pencil $| V | $ is $B$.
\end{lem}
\begin{proof}
Observe that $\dim G_W=2g-4$.
To see that $G_{BL}$  is an irreducible divisor we consider the incidence variety in $C_y\times G_W$ defined by:
\[
 I=\{(p,V)\in C_y\times G_W \mid \text{ $p$ is a base point of the pencil } \vert V \vert \}.
\]
Then the fibers of the first projection are $Grass(2,W(-p))$, all irreducible of dimension $2(g-1-2)=2g-6$. Hence $I$ is irreducible of dimension $2g-5$. Since the second projection map has finite fibres we get that $G_{BL}$ is irreducible of codimension $1$ in $G_W$.
\end{proof}

\vskip 3mm
At this point the proofs of Theorems \ref{main} and \ref{jacs_dom_hyper} diverge. In the rest of this section we finish  the proof of Theorem \ref{main}.

\vskip 3mm
\textbf{Proof of \ref{main}: } We fix a subspace $V$ as in the Claim. We want to see the existence of $\xi \in \mathbb T$ such that  $\xi \cdot V=0$. This means that we look at $\xi $ as an element in $Sym^2 H^{1,0}(A_y)^*$ and hence as a symmetric map $\delta_ \xi:H^{1,0}(A_y)\lra H^{1,0}(A_y)^*$. The subspace $V$ is a subspace of  $W$, while $W$ is isomorphic to  $H^{1,0}(A_y)$, we ask for the condition $\delta_\xi (V)=0$. To prove that such a $\xi $ exists  we consider the restriction to $V$:
 \[
  \delta _{\xi }: V \lra H^{1,0}(A_y)^* \cong W^* =V ^*\oplus V^{*\perp }. 
 \]
This provides an element in $V^*\otimes V^* +V^*\otimes V^{*\perp }$ which, by the symmetry, belongs to $Sym^2 V^* + V^* \otimes V^{*\perp }$. 
This last space has dimension $3+2(g-2)=2g-1$. This is the place where the dimension assumption is used: since $\dim \mathcal Y \ge 2g$ we get that the linear map
\[
\mathbb T \lra Sym^2 V^* + V^* \otimes V^{*\perp }
\]
sending $\xi$ to $\delta_{\xi \mid V} $ has non trivial kernel. Therefore there exists  non-zero $\xi$ vanishing on $V$ as desired.
Let us fix $\xi \ne 0$ such an element for the rest of the proof.
Observe that $\xi$ can be seen via $f$ as an infinitesimal deformation of $C_y$. We denote by $E_{\xi}$ the rank $2$ vector bundle on $C_y$ associated to $\xi$ via the isomorphism $H^ 1(C_y,T_{C_y})\cong Ext^ 1(\omega_{C_y},\mathcal O_{C_y})$. By definition there is a short exact sequence of sheaves:
\[
 0\lra \mathcal O_{C_y} \lra E_\xi \lra \omega_{C_y} \lra 0.
\]
The connecting homomorphism in the associated long exact sequence of cohomology $H^0(C_y,\omega_{C_y})\lra H^1(C_y,\mathcal O_{C_y})$ is the cup-product with $\xi \in H^1(C_y,T_{C_y})$.
Consider two linearly independent holomorphic forms $\omega_1, \omega_2$  generating $V$. Since $\xi \cdot \omega_i=0$, then both $\omega _1, \omega_2$ lift to sections $s_1,s_2\in H^ 0(C_y,E_\xi)$. Now we apply the adjunction procedure as explained in \cite{collino_pirola}: the image of $s_1\wedge s_2 $ by the map
\[
 \Lambda^ 2H^ 0(C_y,E_y) \lra H^ 0(C_y, \Lambda ^ 2 E_\xi)\cong H^ 0(C_y,\omega_{C_y})
\]
provides a new form $adj_\xi(V)$ on the curve which is well-defined up to constant in the quotient of $H^ 0(C_y,\omega_{C_y})/V$. The main property of the adjoint form is the following.
\begin{thm} (\cite[Theorem 1.1.8]{collino_pirola})\label{CoPi}
 Let $\xi \in H^1(C_y,T_{C_y})$ and $V$ as above. Then  $adj_\xi(C_y)$ vanishes (i.e. the image of $s_1\wedge s_2$ in $H^ 0(C_y,\omega_{C_y})$ is contained in $V$) if and only if $\xi$ belongs to the kernel of the natural map:
 \[
  H^1(C_y,T_{C_y}) \lra H^1(C_y,T_{C_y}(B)),
 \]
where $B$ is the base locus of the linear system $| V |$.
\end{thm}

\vskip 3mm
Observe that the holomorphic forms on a hyperelliptic curve are all anti-invariant by the action of the natural involution (since there are no global holomorphic forms on the projective line). On the other hand, since $\xi$ is a deformation preserving the hyperelliptic condition, the involution on $C_y$ extends to $E_{\xi }$ and therefore to $H^0(C_y,E_\xi )$. We can choose anti-invariant liftings $s_1, s_2$ of a basis $\omega_1, \omega_2$ of $V$, hence $s_1\wedge s_2$ is invariant. This implies that the adjoint form is also invariant and therefore 
\[
 adj_\xi (V)=0.
\]
By applying Theorem \ref{CoPi} we know that the image of $\xi $ in 
\[
H^1(C_y, T_{C_y}(B))\cong Ext^1(\omega_{C_y}(-B), \mathcal O_{C_y}) 
\]
 is zero. This says that the corresponding extension is trivial, so  the short exact sequence in the first row of the next diagram splits (i.e. $i^* E_{\xi}=\mathcal O_{C_y} \oplus \omega _{C_y}(-B)$):
 \[
  \xymatrix@C=1.pc@R=1.8pc{
 0 \ar[r] &
 \mathcal  O_{C_y} \ar[r] \ar@{=}[d]   &
i^* E_{\xi} \ar[r] \ar[d]   &
 \omega _{C_y}(-B) \ar[r] \ar@{^{(}->}[d]^{i} &  0 \\
 0 \ar[r]&   O_{C_y} \ar[r]   &  E_{\xi} \ar[r]  & \omega _{C_y} \ar[r]  & 0
 }
 \]
which implies that the connecting homomorphism in the associated long exact sequence of cohomology $H^0(C_y,\omega _{C_y}(-B))\lra H^1(C_y,\mathcal O_{C_y})$ is trivial. Therefore $\xi \cdot H^0(C_y,\omega _{C_y}(-B))=0=\xi \cdot W$, this says that $\xi$ is in the kernel of $d\Phi_y$ which is a contradiction. This completes the proof of \ref{main}. \qed

 \vskip 3mm
 \begin{rem}
  Observe that the moduli space of ppav can be replaced by any moduli space of polarized abelian varieties of some fixed polarization type.
 \end{rem}

 \vskip 3mm
 \begin{rem}\label{example}
  The following example gives a family of hyperelliptic curves dominating a family of polarized abelian varieties. For any $C$ hyperelliptic of genus $g$, fix the two-to-one  map $C\lra \mathbb P^1$ and a degree $2$ covering $\mathbb P^1 \lra \mathbb P^1$ ramified in $2$ points. The fibre product $\tilde C:=C\times_{\mathbb P^1}\mathbb P^1$ is a curve of genus $2g+1$ with a degree $2$ map $\tilde C \lra C$. Then the Jacobian $J\tilde C$ dominates the $g$-dimensional Prym variety $P(\tilde C,C)$. By moving $C$ in the hyperelliptic locus we get a family of polarized abelian varieties of dimension $2g-1$ dominated by hyperelliptic Jacobians. So the bound $\dim \mathcal Y \ge 2g$ in Theorem \ref{main} is sharp.
\end{rem}

\vskip 5mm
\section{Jacobians dominated by hyperelliptic Jacobians}

The goal of this section is to finish the proof of Theorem \ref{jacs_dom_hyper}.  We keep from the beginning of the last section  the notation
\[
 \xymatrix@C=1.pc@R=1.8pc{
\mathcal {JC}  \ar[rd] \ar[rr]^{f}  && \mathcal {JD} \ar[dl] \\
&\mathcal U &
}
\]
where $\mathcal {JC}$ is a family of hyperelliptic Jacobian and $\Phi: \mathcal U \lra \mathcal M_g $ is the generically finite dominant map $\mathcal U \lra \mathcal Y\subset \mathcal M_g$ associated with the family $\mathcal {JD}$. Moreover, for a general point $y\in \mathcal U$ we have $\mathbb T:= T_{\mathcal U}(y)\cong T_{\mathcal Y}(\Phi(y)) $ and the inclusion
$ \mathbb T \hookrightarrow Sym^2\, H^0(D_y,\omega_{D_y})^*$. We also keep the notation $W$ for the identification of  $H^0(D_y,\omega_{D_y})$ as a vector subspace of $H^0(C_y,\omega_{C_y})$, $B$ for the base locus of the linear system 
$| W | $, and  $G_W$ for the Grassmannian of $2$-dimensional vector spaces in $W$. Remember (see Lemma \ref{grass}) that the set of $2$-dimensional subspaces of $W$ with some additional base point defines an irreducible divisor $G_{BL}$ of $G_W$.

\vskip 3mm
We want to compare this irreducible divisor with the closed subset of $G_W$  given by the subspaces that deform in some direction tangent to $\mathcal Y$:
\[
 \mathcal D=\{V\in G_W \mid \text{there exists } \xi \in \mathbb T \text{ such that } \xi \cdot V=0 \}.
\]
We claim that either $\mathcal D$ is a divisor or $\mathcal D=G_W$. This is the statement of the next lemma: 

\vskip 3mm
\begin{lem}
 With the above notation $\mathcal D$ has dimension $\ge 2g-5$.
\end{lem}

\begin{proof}
 We consider the Grassmannian $G_W$ embedded in a projective space $\mathbb P^N$ via the Pl\"ucker embedding. First we see that $\mathcal D$ intersects any line of $\mathbb P^N$ contained in $G_W$. Note that $V_1,V_2\in G_W$ generates a line contained in $G_W$ if and ony if $V_1\cap V_2\neq \emptyset $. Hence we can assume that $V_1=\langle \omega, \omega_1 \rangle $ and $V_2=\langle \omega, \omega_2 \rangle$. The elements of the line $r:V_1 \vee V_2 \subset G_W$ are the subspaces $\langle \omega, \lambda \omega_1 + \mu \omega_2 \rangle $. Observe that the multiplication by $\omega $ defines a linear map:
 \[
  \mathbb T \lra W^*, \, \, \xi \lra \xi(\omega).
 \]
Hence $K_{\omega}=\{\xi \in \mathbb T \mid \xi(\omega) \}$ is a linear subspace of dimension at least $g-1$. By the symmetry of the map $\cdot \xi: W\lra W^*$ we get that $\xi \in K_\omega $ induces  a linear map:
\[
 \cdot \xi: W/\langle \omega \rangle \lra (W/\langle \omega \rangle)^*.
\]
Let us consider the projective line $\mathbb P^1$ given by the projectivization of the image of the subspace $\langle \omega_1, \omega_2 \rangle$ in $W/\langle \omega \rangle $. This line corresponds to the line $r$ considered at the beginning of the proof. Then there is a map of bundles of rank $g-1$ on this $\mathbb P^1$:
\[
 K_{\omega } \otimes \mathcal O_{\mathbb P^1}(-1) \cong \mathcal O_{\mathbb P^1}(-1)^{\oplus (g-1)} 
 \lra  (W/\langle \omega \rangle)^* \otimes \mathcal O_{\mathbb P ^1}\cong \mathcal O_{\mathbb P^1}^{\oplus (g-1)}. 
\]
Hence there is some point where the map drops the rank, so there exists $\xi $ killing $\omega $ and a form $\lambda \omega _1 + \mu \omega_2$ for some $(\lambda: \mu)\in \mathbb P^1$. Hence $\mathcal D$ intersects the line $r$.

To finish the proof of the lemma it is enough to show that a closed subvariety $\mathcal D\subset G_W$ intersecting all the lines contained in $G_W$ has codimension at most one. This is an easy projective geometry argument that we add due to the lack of a reference in the literature: we proceed by induction on $g$, for $g=3$ we have $G_W=\mathbb P(W)^*\cong \mathbb P^2$
and the statement is obvious. Assume that $g\ge 4$ and let $L$ be a linear subspace of codimension $2$ in $\mathbb P(W)$, then the Schubert cycle of the lines intersecting $L$ is an effective divisor $R_L\subset G_W$. Notice that for any hyperplane $H\subset \mathbb P(W)$ containing $L$, the Grassmannian of lines in $H$, $G_H$, is contained in $R_L$ and has codimension $2$ in $G_W$:
\[
G_H\subset R_L \subset G_W. 
\]
Observe that  $\mathcal D \cap G_H$ satisfy the hypothesis on the intersection on the lines on $H$ hence by induction hypothesis $\mathcal D \cap G_H$ is a divisor (or $G_H$) in 
$G_H$. Moving $H$ in the pencil of hyperplanes through $L$ we obtain that $\mathcal D\cap R_L$ is a divisor on $R_L$: indeed, otherwise $\mathcal D\cap R_L =\mathcal D \cap G_H$ 
for any $H$ in the pencil, therefore $\mathcal D\cap R_L =\mathcal D \cap (\bigcap _{H\supset L}\ G_H)$ which is a subset of the Grassmannian of lines in $L$, $G_L$. This contradicts that $\mathcal D\cap G_H$ is a divisor in $G_H$ since $\dim G_L=\dim G_H -2$.

Finally, if the codimension of $\mathcal D$ in $G_W$ were at least $2$ then it would be contained in $R_L$ for any $L$ which is impossible since the intersection of all the $R_L$ is empty. This finishes the proof. 
\end{proof}

If $\mathcal D$ is not contained in $G_{BL}$ we obtain that there exists $\xi \in \mathbb T$ and a subspace $V$ with base locus $B$ such that $\xi \cdot V=0$. As in the proof of Theorem \ref{main} we can apply the adjunction procedure to get the adjoint form, which must be trivial since it is an invariant form on a hyperelliptic curve. Therefore $\xi \cdot W=0$ which contradicts the injectivity of $d\Phi (y)$.

\vskip 3mm
So, from now on, we can assume that $\mathcal D \subset G_{BL}$ and since the second divisor is irreducible $\mathcal D = G_{BL}$. This means that for any point $p\in C_y$ and for any $V\subset W(-p)$ there exists a deformation $\xi$ killing $V$. Let us consider the map $\gamma: C_y\setminus B \lra |W|^*\cong \mathbb P^{g-1}$ given by  the linear system $|W|$. This is a linear projection composed with the canonical map of $C_y$. In particular, since $C_y$ is hyperelliptic, $\Gamma:= \gamma (C_y)$ is a rational curve. Arguing as in section $3$ it is easy to prove that for a general $2$-dimensional linear subspace $V$ of $W(-p)$ the base locus of the pencil $| V | $ is $\gamma^{-1} (\gamma (p))$.

\vskip 3mm
We fix $\xi\in T_{\mathcal Y}(\Phi(y))$ with $\xi \cdot V=0$ for a pencil with base locus exactly in $\gamma ^{-1}(\gamma (p))$. We can still do the construction of the adjoint form but instead of a contradiction we obtain that $\xi$ kills the whole subspace $W(-\gamma^{-1} (\gamma (p)))$ of dimension $g-1$. This means that $\xi$, as a infinitesimal deformation of $D_y$, has rank $1$. Thus we find a map from $C_y$ to the locus of rank $1$ deformations $S$ of $D_t$ in $\mathbb P H^1(D_t, T_{D_t})$ which  
factorizes through $\gamma$, hence  the image is a rational curve. By \cite[p. 275]{Gr_IVHS_III}, if  the Clifford index of $D_t$ is at least $2$ then $S$ is  the bicanonical image of $D_t$, hence we obtain a contradiction. Therefore the Clifford index is at most $1$. Moreover 
$D_y$ is not a quintic plane curve since the main result in \cite{quintic} states that there are no infinitesimal deformations of rank $1$ of a quintic plane curve that preserves the planarity of the curve. The conclusion is that eiher $D_y$ is hyperelliptic or it is trigonal, as claimed.

\end{document}